\documentclass[11pt]{amsart}
\usepackage{bbm}
\usepackage{mathrsfs}
\usepackage{cases}
\usepackage{latexsym}
\usepackage[all]{xy}
\usepackage{stmaryrd}
\usepackage{amsfonts}
\usepackage{float}
\usepackage{amssymb,amscd,amsthm,dsfont}
\usepackage{amsmath}
\usepackage{fancyhdr}
\usepackage{amsxtra,ifthen}
\usepackage{verbatim}
\usepackage{tikz-cd}
\tikzstyle{dot}=[draw, fill =black, circle, inner sep=0pt, minimum size=2pt]

\theoremstyle{plain}

\numberwithin{equation}{section}
\newtheorem{theorem}{Theorem}[section]
\newtheorem{lemma}[theorem]{Lemma}

\newtheorem{corollary}[theorem]{Corollary}

\newtheorem{definition}[theorem]{Definition}
\newtheorem{example}[theorem]{Example}

\begin{document}

\title[Semisimplicity of affine cellular algebras]
{Semisimplicity of affine cellular algebras}

\thanks {Corresponding Author: Bowen Sun}
\thanks {The work is supported by the Natural Science Foundation of Hebei Province, China (A2021501002); China Scholarship Council (202008130184)  and
Natural Science Foundation of China (11871107).}

\author{Yanbo Li}

\address{Li: School of Mathematics and Statistics, Northeastern
University at Qinhuangdao, Qinhuangdao, 066004, P.R. China}

\email{liyanbo707@163.com}

\author{Bowen Sun}

\address{Sun: Department of mathematics, School of Science, Northeastern
University, Shenyang, 110819, P.R. China}

\email{bowen$_{-}$un@163.com}

\begin{abstract}
In this note, we prove that an affine cellular algebra $A$ is semisimple if and only if
the scheme associated to $A$ is reduced and 0-dimensional, and the bilinear forms with respect to all layers of $A$ are isomorphisms.
Moreover, if the ground ring is a perfect field, then $A$ is semisimple if and only if it is separable.
We also give a sufficient condition for an affine cellular algebra being Jacobson semisimple.
\end{abstract}



\maketitle

\section{Introduction}

Affine cellular algebras were introduced by Koenig and Xi in \cite{KX3},
who extended the framework of cellular algebras due to Graham and Lehrer \cite{GL}.
The isomorphism classes of simple representations of
affine cellular algebras are parameterized by the complement of finitely many subvarieties in a
finite disjoint union of affine varieties. Many classes of algebras have been
found to be affine cellular, including affine Hecke algebras of type A \cite{KX3}, affine Temperley-Lieb algebras \cite{KX3},
Khovanov-Lauda-Rouquier algebras of finite types \cite{KL1, KL2},  affine quasi-hereditary algebras \cite{K}, affine Brauer algebras \cite{C1},
affine Birman-Murakami-Wenzl algebras \cite{C2},
affine Yokonuma-Hecke algebras \cite{C5}, affine q-Schur algebras \cite{C4}, BLN algebras \cite{C3}, affine Hecke algebras of rank two \cite{GM} and so on.

\smallskip

In \cite{CKL}, Carvalho, Koenig and Lomp studied ring theoretical structure of affine cellular algebras.
The aim of this note is to do some work along this direction. More exactly, we will study the semisimplicity of affine cellular algebras.
Let us first clarify some relevant notions.

Recall that a ring $R$ is called semisimple if it is semisimple as a left $R$-module. Every semisimple ring is isomorphic
to a finite direct product of matrix rings and is left and right artinian.
Denote by $J(R)$ the Jacobson radical of $R$. Then ring $R$ is called Jacobson semisimple if $J(R)=0$. Note that the action of $J(R)$ on a simple $R$-module is zero.
A ring $R$ is called semiprime if $R$ has no nonzero nilpotent ideals.
We have the following hierarchy of the algebras aforementioned:
$$\{\rm semisimple\}\subset \{\rm Jacobson\, semisimple\}\subset \{\rm semiprime\}$$
It is helpful to point out that if $R$ is a left artinian ring,
then $R$ being semisimple is equivalent to being Jacobson semisimple, and is equivalent to being semiprime, too.
Note that a commutative semiprime ring is also called reduced.

\smallskip

The main goal of this note is to prove the following theorem, which gives a sufficient and necessary condition for an affine cellular algebra $A$ to be semisimple.
Denote by $Spec[A]$ the associated scheme to $A$ (see Definition \ref{4.4}) and denote by $\phi_j$ the bilinear form of the $j$-th layer of $A$.

\smallskip

\noindent {\bf Theorem}\, {\em Let $K$ be a field and let $A$	be an affine cellular $K$-algebra. Then $A$ is semisimple if and only if
\begin{enumerate}
\item[(1)] $Spec[A]$ is a reduced $0$-dimensional scheme;
\item[(2)] $\phi _j$ are invertible for all $j$.
\end{enumerate}}

\smallskip

Moreover, if the ground ring is a perfect field, then $A$ is semisimple if and only if it is separable.
We also give a sufficient condition for an affine cellular algebra being Jacobson semisimple.

\bigskip

\section{Affine cellular algebras}

In this section, we give a quick review on the definitions and some known results
about an affine cellular algebra which are needed in the next section. The main reference is \cite{KX3}.

\smallskip

Let $K$ be a principal ideal domain. Given two
$K$-modules $V$ and $W$, we denote by $\tau$ the swich map: $V\otimes_KW\rightarrow W\otimes_KV$, $v\otimes_Kw\rightarrow w\otimes_Kv$ for $v\in V$ and $w\in W$. A $K$-algebra $B$ is called affine if $B=K[x_1, \cdots, x_t]/I$, where $K[x_1, \cdots, x_t]$ is a polynomial ring in finitely many variables $x_1, \cdots, x_t$ and $I$ an ideal. A $K$-involution $\ast$ on a $K$-algebra $A$ is a $K$-linear anti-automorphism with $(a^{\ast})^{\ast}=a$ for all $a\in A$.

\begin{definition}\cite[Definition 2.1]{KX3}\label{2.1}	
Let $A$ be a unitary $K$-algebra with a $K$-involution $\ast$. A two-sided ideal $J$ in $A$ is called an affine cell ideal if and only if the following data are given and the following conditions are satisfied:\vskip2mm
\begin{enumerate}
\item[(1)] The ideal $J$ is fixed by $\ast$: $(J)^{\ast}=J.$
\smallskip
\item[(2)]There exist a free $K$-module $V$ of finite rank,  an affine $K$-algebra $B$ with identity and with a $K$-involution $\sigma$ such that $\Delta :=V\otimes_{K} B$ is an $A$-$B$-bimodule, on which the right $B$-module structure is induced by $B_{B}$.
\smallskip
\item[(3)] There is an $A$-$A$-bimodule isomorphism $\alpha :J\rightarrow \Delta\otimes_{B}\Delta',$ where $\Delta'=B\otimes_{K}V$ is a $B$-$A$-bimodule with the left $B$-module induced by ${}_{B}B$ and with the right $A$-module structure defined by $(b\otimes v)a :=\tau(a^{\ast}(v\otimes b))$ for $a\in A$, $b\in B$ and $v\in V$, such that the following diagram is commutative:
\[\begin{CD}
J   @>\alpha>>\Delta\otimes_{B}\Delta'\\
@V \ast VV                  @VVv_1\otimes b_1\otimes_{B}b_2\otimes v_2\mapsto v_2\otimes \sigma(b_2)\otimes_{B}\sigma(b_1)\otimes v_1V\\
J         @>\alpha>>   \Delta\otimes_{B}\Delta'
\end{CD}\]
\end{enumerate}
\smallskip
The algebra $A$ with $K$-involution $\ast$ is called affine cellular if and only if there is a $K$-module decomposition $A=J_{1}'\oplus J_{2}'\oplus\cdots J_{m}'$ (for some $m$) with $(J_{j}')^{\ast}=J_{j}'$ for each $j$ $(j=1,\dots,m)$ and such that setting $J_{j}: =\bigoplus_{l=1}^{j}J_{l}'$ gives a chain of two-sided ideals of $A$: $0=J_{0}\subset J_{1}\subset J_{2}\subset\cdots\subset J_{m}=A$ (each of them fixed by $\ast$), and each $J_{j}'=J_{j}/J_{j-1}$  is an affine cell ideal of $A/J_{j-1}$ (with respect to the involution induced by $\ast$ on the quotient).\vskip3mm	
\end{definition}

Clearly, if all the affine algebras $B_j$ are equal to the ground ring $K$, we recover the definition of a cellular algebra given by Koenig and Xi in \cite{KX2}.
Note that the original definition of a cellular algebra was given by Graham and Lehrer in \cite{GL}.

\smallskip

For an affine cell ideal $J$ in an algebra $A$, the following lemma gives the basic structure of $J$ when viewed as an algebra (without unit) in itself.

\begin{lemma}\cite[Proposition 2.2]{KX3}\label{2.2}
Let $J$ be an affine cell ideal in a $K$-algebra $A$ with an involution $\ast$. We identify $J$ with
$\Delta \otimes_B \Delta' = V \otimes_K B \otimes_K V$. Then:
\begin{enumerate}
\item[(1)]\, There is a $K$-linear map $\phi : V \otimes_K V \rightarrow B$ such that
$$(u \otimes b \otimes v)(u' \otimes b' \otimes v') = u \otimes b\phi(v, u')b' \otimes v'$$
for all $u, u', v, v' \in V$ and $b, b' \in B$.
\item[(2)]\, If $I$ is an ideal in $B$ and $u, v \in V$, then $V \otimes_K I \otimes_K V$ is an ideal in $A$.
\end{enumerate}
\end{lemma}

Because of the importance of bilinear form $\phi$, we often write $\Delta\otimes_B \Delta'$ as $\mathcal{A}(V, B, \phi)$.
Let $\{v_1, \cdots, v_n\}$ be a basis of $V$ and identify the bilinear form $\phi$ with the matrix $\phi=(\phi_{ij})=\phi(v_i, v_j)$
(we often use the same notation for the bilinear form and its matrix in this note).
Then $J$ is isomorphic to a swich algebra $S(M_n(B),\, (\phi_{ij}))$ with the definition given as follows.

\begin{definition} \cite[Definition 3.3]{KX3}\label{2.3}
Let $\Lambda$ be a $K$-algebra and fix an element $a_0\in \Lambda$. We define a new $K$-algebra $\widetilde{\Lambda}= S(\Lambda,\, a_0)$,
called the swich algebra of $\Lambda$ with respect to $a_0$, where as a set $\widetilde{\Lambda}=\{\widetilde{a}\mid a\in \Lambda\}$,
and the algebra structure on $\widetilde{\Lambda}$ is given by
$$\widetilde{a} + \widetilde{b} = \widetilde{a + b}, \quad\quad a, b \in \Lambda,$$
$$\widetilde{a}\cdot\widetilde{b} = \widetilde{a a_0 b},\,\, \quad\quad a, b \in \Lambda,$$
$$\lambda\widetilde{a}=\widetilde{\lambda a}, \quad\quad \lambda\in K,\, a\in \Lambda.$$
\end{definition}

A swich algebra of a matrix algebra $M_n(B)$ is in fact a generalized matrix algebra in the sense of \cite{B}.
By a straightforward computation, one can prove that the map $\varphi: \widetilde{\Lambda}\rightarrow \Lambda$
defined by $$\widetilde{a}\mapsto \varphi(\widetilde{a})=aa_0$$ is an algebra homomorphism, and
a $\Lambda$-module $M$ will become a $\widetilde \Lambda$-module $M^{\varphi}$ via $\varphi$.
The following lemma establishes a relationship between the set of all simple modules over $\Lambda$ and that over a swich algebra $\widetilde{\Lambda}$.

\begin{lemma}\cite[Theorem 3.10]{KX3}\label{2.4}
Let $\Lambda$ be a $K$-algebra with identity such that $\Lambda$ is finitely generated over its centre.
Then there is a bijection $\omega$ between the set of non-isomorphic simple $\Lambda$-modules $E$ with $\widetilde{\Lambda}.E \neq 0$,
and the set of all non-isomorphic simple $\widetilde{\Lambda}$-modules, which is given by
$E\mapsto E^\varphi/\{x \in E^\varphi \mid \widetilde{\Lambda}.x = 0\}$.
\end{lemma}

We conclude this section by a result of \cite{KX3} which is needed in Section 3.

\begin{lemma}\cite[Theorem 3.12 (2)]{KX3}\label{2.5}
Let $A$ be an affine cellular algebra with a cell chain $0=J_{0}\subset J_{1}\subset J_{2}\subset\cdots\subset J_{m}=A$ such that
each layer $J_{j}/J_{j-1}\cong \mathcal{A}(V_j, B_j, \phi_j)$. Then for $1\leq j\leq m$, $\phi_j$ is an isomorphism if and only if
the determinant $det(\phi^{(j)}_{st} )$ of $\phi_j$ is a unit in $B_j$. In particular, if all $\phi_j$ are isomorphisms,
then $A$ is isomorphic, as an affine cellular $K$-algebra, to $\bigoplus_{j=1}^m M_{n_j}(B_j)$, where $n_j$ is the dimension of $V_j$.
\end{lemma}

The algebra $\bigoplus_{j=1}^m M_{n_j}(B_j)$ will be called the asymptotic algebra of the
affine cellular algebra $A$.

\bigskip

\section{Semisimplicity}

In this section, we study the semisimplicity of affine cellular algebras.
We first need to review some definitions and notations from commutative algebra.
The main references here are \cite{Ke} and \cite{U}.

\begin{definition}\label{4.1}
Let $R$ be a commutative ring with identity. The spectrum of $R$ is defined to be
$Spec(R):=\{\mathfrak{P}\subseteq R\ |\ \mathfrak{P}\ is\ a \ prime\ ideal \}$.
\end{definition}

It is well-known that one can put a topology on $Spec(R)$, which is the so-called Zariski topology.
Then we can give the definition of an affine scheme.

\begin{definition}\label{4.3}
An affine scheme is a pair $(Spec(R), \mathscr O_R)$ consisting a spectrum of $R$
equipped with Zariski topology, and its structure sheaf $\mathscr O_R$. If $R$ is a reduced ring,
then $(Spec(R), \mathscr O_R)$ is called reduced.
The dimension of $(Spec(R), \mathscr O_R)$ is defined to be the Krull dimension of $R$.
\end{definition}

We refer the reader to \cite[Definition 2.20]{U} for the definition of the structure sheaf mentioned in Definition \ref{4.3}. By abuse of notations,
we will often write simply $Spec(R)$ for the affine scheme $(Spec(R), \mathscr O_R)$.

\smallskip	

Based on the above preparation, we can define the associated scheme to an affine cellular algebra, which will play a key role in this section.

\begin{definition}\label{4.4}
Let $A$ be an affine cellular algebra. We call $$Spec[A]:=Spec(\prod_{j=1}^{m}B_j)$$ the associated scheme to $A$.
\end{definition}	

Given an affine cellular algebra $A$, we will show that some ring theoretical properties of $A$,
for example, semisimplicity, are partially determined by  $Spec[A]$.

Let us first consider Jacobson semisimplicity. For this goal, we need the following lemma about a swich algebra.
For simplicity of description, we stipulate that $0$ is a zero-divisor.
	
\begin{lemma}\label{4.5}
Let $R$ be a unital ring such that $R$ is finitely generated over its center and $\widetilde R=S(R, a_0)$, a swich algebra of $R$.
Then $\widetilde R$ is Jacobson semisimple if and only if $R$ is Jacobson semisimple with $a_0$ not a zero-divisor.
\end{lemma}

\begin{proof}
By Lemma \ref{2.4},  there is a bijection between the set of
non-isomorphic simple $R$-modules $E$ with $\widetilde R.E\neq 0$, and the set of all non-isomorphic simple $\widetilde R$
-modules, which is given by $E\mapsto E^{\varphi}/N$, where $N$ is $\{x\in E\mid\widetilde R.x=0\}$.
For each  maximal left ideal ${\mathfrak m}$ of $R$, denote the corresponding simple $R$-module $R/{\mathfrak m}$ by $E_{\mathfrak m}$.

($\Rightarrow$) Suppose that $\widetilde R $ is Jacobson semisimple. Then $\widetilde R $ is semiprime.
Consequently, we have from \cite[Lemma 2.1]{CKL} that $a_0$ is not a zero-divisor. Take an element $a$ of $J(R)$.
Then $aE_{\mathfrak m}=0$ for arbitrary $\mathfrak m$. Assume that $\widetilde R.E_{\mathfrak m}\neq 0$.
For $x+N_{\mathfrak m}\in E_{\mathfrak m}^{\varphi}/N_{\mathfrak m}$, we have
$$\widetilde{a}.(x+N_{\mathfrak m})=\widetilde{a}.x+N_{\mathfrak m}=aa_0.x+N_{\mathfrak m}=0+N_{\mathfrak m},$$
where the last equality holds because $a_0.x\in E_{\mathfrak m}$ and $aE_{\mathfrak m}=0$.
This implies that $\widetilde{a}$ annihilates all simple $\widetilde{R}$-modules, that is,
$\widetilde{a}\in J(\widetilde{R})$, and hence $\widetilde{a}=0$. As a result, $a=0$ and $R$ is Jacobson semisimple.
		
($\Leftarrow$) Let $R$ be a Jacobson semisimple ring. Suppose that $\widetilde{a}\in J(\widetilde{R})$ and $\mathfrak m$ an arbitrary maximal left ideal of $R$.
Then $\widetilde{a}.(E_{\mathfrak m}^{\varphi}/N_{\mathfrak m})=0$.
Note that $R$ is a unital ring, and it is clear that $1_R\notin \mathfrak m$.
For arbitrary $x\in R$, denote $x+\mathfrak m\in E_{\mathfrak m}$ by $[x]$.
We have $$0=\widetilde{a}.([1]+N_{\mathfrak m})=\widetilde{a}[1]+N_{\mathfrak m}=[aa_0]+N_{\mathfrak m},$$ and we deduce
$[aa_0]\in N_{\mathfrak m}$, or $\widetilde R.[aa_0]=0$ in $E_{\mathfrak m}$. In particular, $\widetilde{1}.[aa_0]=0$, and thus $a_0aa_0\in \mathfrak m$.
As a result, $a_0aa_0\in J(R)$ because of $\mathfrak m$ being arbitrary. Now the Jacobson semisimplicity of $R$ forces $a_0aa_0$ to be zero.
Combining this result with the fact that $a_0$ is not a zero-divisor yields $a=0$. Consequently, the radical of $\widetilde{R}$ is zero and this completes the proof.
\end{proof}

Now we can give a sufficient condition for an affine cellular algebra being Jacobson semisimple.

\begin{theorem}\label{4.6}
Let $A$ be an affine cellular algebra. Then $A$ is Jacobson semisimple if
\begin{enumerate}
\item[(1)] $Spec[A]$ is a reduced scheme;
\item[(2)] all $\phi _j$ are not zero-divisors.
\end{enumerate}
\end{theorem}

\begin{proof}
The affine cellularity of $A$ implies that a layer $J_j/J_{j-1}$ is isomorphic to
a swich algebra $\widetilde M_{n_j}(B_j)=S(M_{n_j}(B_j), \,\phi_j)$.
Since $Spec[A]$ is a reduced scheme, each $B_j$ is a reduced ring. Note that a reduced affine algebra is Jacobson semisimple.
Then $M_{n_j}(B_j)$ is Jacobson semisimple because the matrix algebra over a Jacobson semisimple ring is Jacobson semisimple too.
Note that $\phi_j$ is not a zero-divisor. Then by Lemma \ref{4.5}, we have $\widetilde M_{n_j}(B_j)$ is semisimple.

Take an element $r \in J(A)$ and assume that  $r=\sum_{j = 1}^{m} r_j$, where $r_j\in J_j'$.
Then $r$ annihilates all simple modules. According to the representation theory of affine cellular algebras \cite[Theorem 3.12]{KX3},
the actions of $r_i, \,\,i=1,2,\dots m-1$ on simple modules belonging to the top layer are all zeros.
Thus $r_m$ annihilates all simple modules of the top layer. Now the Jacobson semisimplicity of each layer proved above forces $r_m$ being zero.
By continuing this process finitely many times, we obtain $r_i=0$, for $i=1,2,\dots, m$, that is, $r=0$. This completes the proof.
\end{proof}

\begin{corollary}\label{4.7}
For a unital affine cellular algebra $A$, if $Spec[A]$ is a reduced scheme and $\phi _j$ are not zero-divisors for all $j$, then A is semiprime.
\end{corollary}

The following is an example where the sufficient criterion given in Theorem \ref{4.6} is not necessary.
The example also implies that it is likely far away from a characterisation of Jacobson semisimplicity.

\begin{example}
Let $F$ be a field and let $K$ be the formal power series ring $F[[x]]$.
Then the Jacobson radical of $K$ is the ideal generated by $x$. This implies that $K$ is not Jacobson semisimple.
Let $A=K[x]$. Then $A$ is Jacobson semisimple (\cite{H} Page 433 Exercise 14 (c)). We claim that $A$ is an affine cellular algebra.
In fact, we can take a cell chain of $A$ to be $0\subset (x)\subset A$ and define $B_1=A$, $V_1$ to be the free $K$-module with basis $\{x\}$ and $B_2=K$, $V_2=K$.
Since $B_2=K$ is not a reduced ring, $Spec[A]$ is not a reduced scheme.
\end{example}

Let us begin to study semisimplicity of an affine cellular algebra over a field $K$ now. As is well-known, a semisimple algebra is left artinian.
So we first give a sufficient and necessary condition for an affine cellular algebra to be left artinian as follows.

\begin{lemma}\label{4.8}
Let $A$ be an affine cellular algebra. Then $A$ is a left artin ring if and only if $Spec[A]$ is a $0$-dimensional scheme.
\end{lemma}

\begin{proof}
If $Spec[A]$ is a $0$-dimensional scheme, then the Krull dimension of every $B_j$ is zero.
This is equivalent to that all $B_j$  are finite dimensional $K$-algebras by \cite[Theorem 5.11]{Ke},
and hence the affine cellular algebra $A$ is a finite dimensional $K$-algebra. So $A$ is left artinian.
		
Conversely, assume that $A$ is a left artin ring.
We claim that for arbitrary $i\in \{1, \cdots, m\}$, $B_i$ is an artin ring. In fact, we have from $A$ being left artinian
that $A/J_{i-1}$ is a left artin ring and $J_i'$ is a left artin $A/J_{i-1}$-module.
If $B_i$ is not artinian, then there exists an infinite descending chain of ideals of $B_i$
$$I_1  \supset I_2  \supset I_3  \supset \cdots  \supset I_n  \supset \cdots$$
As a result, we obtain an infinite descending chain of submodules of $J_i'$
$$ V_1 \otimes I_1\otimes V_1  \supset V_1 \otimes I_2\otimes V_1  \supset V_1 \otimes I_3\otimes V_1  \supset \cdots  \supset V_1 \otimes I_n\otimes V_1  \supset \cdots$$
due to Lemma \ref{2.2}. It is a contradiction. Then $Spec[A]$ is a $0$-dimensional scheme, which follows from all $B_i$ being artinian.
\end{proof}

Employing \cite[Theorem 5.11]{Ke} again, we get a direct corollary of Lemma \ref{4.8} as follows.

\begin{corollary}\label{4.9}
Let $A$ be an affine cellular algebra. Then the following statements are equivalent.
\begin{enumerate}
\item[(1)] $A$ is a left artin algebra.
\item[(2)] $Spec[A]$ is a $0$-dimensional scheme.
\item[(3)] $A$ is a finite dimensional $K$-algebra.
\item[(4)] Each $B_j$ is a finite dimensional $K$-algebra.
\end{enumerate}
\end{corollary}

Now we can give a sufficient and necessary condition for an affine cellular algebra being semisimple.

\begin{theorem}\label{4.10}
Let $K$ be a field and let $A$	be an affine cellular $K$-algebra. Then $A$ is semisimple if and only if
\begin{enumerate}
\item[(1)] $Spec[A]$ is a reduced $0$-dimensional scheme;
\item[(2)] $\phi _j$ is invertible for each $j\in\{1,2,\dots m\}$.
\end{enumerate}
\end{theorem}
\begin{proof}
$``\Rightarrow"$ Let $A$ be semisimple. Then $A$ is both left artinian and semirpime, and thus $Spec[A]$ is a
$0$-dimensional scheme by Corollary \ref{4.9}.
As is well-known, the quotients of a semisimple ring are semisimple too. This implies the semisimplicity of $A/J_{j-1}$.
On the other hand, an ideal of a semisimple algebra is semisimple. This gives that $J_j'$ is semisimple since $J_j'$ is an ideal in $A/J_{j-1}$.
As a result, $B_j$ is a reduced ring with $\phi _j$ not a zero-divisor by \cite[Proposition 2.5]{CKL}.
Moreover, we also have from  Corollary \ref{4.9} that every $B_j $ is a finite dimensional $K$-algebra.
Recall that a finite dimensional unital algebra enjoy a special property \cite[Theorem 1.2.1]{DK}: every element is either invertible or a zero-divisor.
This forces $\phi_j$ to be invertible.

$``\Leftarrow"$ If $Spec[A]$ is a reduced 0-dimensional scheme and all $\phi_j$ are invertible,
then combining Corollary \ref{4.7} with Corollary \ref{4.9} implies that $A$ is a finite dimensional semiprime algebra, and consequently, $A$ is semisimple.
\end{proof}

Note that if $B_j$ is a finite dimensional affine $K$-algebra, then $\phi_j$ is invertible if and only if $det(\phi_j)$ is a unit in $B_j$.
Moreover, employing Lemma \ref{2.5} yields an easy result as follows.

\begin{corollary}\label{4.11}
Let $A$ be an affine cellular algebra $A$. Then $A$ is semisimple if and only if it is isomorphic to its asymptotic
algebra and $Spec[A]$ is a reduced 0-dimensional scheme.
\end{corollary}

\smallskip

To conclude the investigation of semisimplicity of an affine cellular algebra,
we enhance the condition ``reduced" in Theorem \ref{4.10} to ``geometrically reduced",
which corresponds to a strengthened version of semisimple algebras: separable algebras.

Let us recall the definitions of a geometrically reduced ring and a separable algebra first.

\begin{definition}\label{4.12}
Let $K$ be a field and $\overline{K}$ the algebraic closure of $K$.
An affine $K$-algebra $A$ is called geometrically reduced if $A\otimes \overline{K}$ is a reduced ring.
A $K$-algebra $A$ is said to be separable if for arbitrary finite extension field $F$ over $K$,
$A\otimes F$ is a semisimple $F$-algebra.
\end{definition}

The definition of a separable algebra can be considered as a certain property which can be preserved under base change.
Especially, we will find that the base change of an affine cellular algebra is actually the base change of
affine algebras $B_j$. In order to study separable affine cellular algebras, we recall the definition of  \'{E}tale algebras.

\begin{definition}\label{4.13}\cite[Definition 1.5.3]{Sz}
A finite dimensional $K$-algebra is said to be \'{E}tale if it is isomorphic to a finite direct sum
of separable extensions of $K$.
\end{definition}

The following lemma can be viewed as an equivalent definition of an \'{E}tale algebra, which implies that
an \'{Etale} algebra is in fact a finite dimensional commutative separable algebra.

\begin{lemma}\cite[Proposition 1.5.6]{Sz}\label{4.14}
Let $A$ be an finite dimensional commutative $K$-algebra. Then the following statements are equivalent.
\begin{enumerate}
\item[(1)] $A$ is \'{E}tale.
\item[(2)] $A\otimes_K\overline{K}$ is reduced.
\end{enumerate}
\end{lemma}

We can give some necessary and sufficient conditions for an affine cellular algebra being separable.

\begin{corollary}\label{4.15}
Let $A$ be an affine cellular $K$-algebra. Then the following statements are equivalent.
\begin{enumerate}	
\item[(1)] A is a separable algebra.
\item[(2)] $Spec[A]$ is a geometrically reduced $0$-dimensional scheme and $det(\phi _j)$ invertible for all $j$.
\item[(3)] $\prod_{j=1}^{m}B_j$ is an \'{E}tale algebra algebra and for all $j$, $det(\phi _j)$ invertible.
\item[(4)] For all $j$, $B_j$ are \'{E}tale algebra algebra and $det(\phi _j)$ invertible.
\end{enumerate}
\end{corollary}

\begin{proof}
It follows from Lemma \ref{4.14} that (2) is equivalent to (3), and the equivalence between (3) and (4) is clear.
Then we only need to prove $(1)\Leftrightarrow (2)$.

$(1)\Rightarrow (2)$ Let $A$ be separable. Then by Definition \ref{4.12}, $A$ is semisimple.
As a result, $Spec [A]$ is a reduced 0-dimensional scheme  with $det(\phi _j)$ invertible by Theorem \ref{4.10}.
Then we only need to prove $Spec[A]$ is geometrically reduced, or $\prod_{j=1}^{m}B_j\otimes \overline{K}$ is reduced.
This is clear from the semisimplicity of $A\otimes \overline{K}$.

$(2)\Rightarrow (1)$ Assume that $Spec[A]$ is a geometrically reduced scheme with $det(\phi _j)$ invertible. We deduce by Lemma \ref{4.14} that
every $B_j$ is an \'{E}tale algebra and hence a separable algebra.
This implies that $B_j\otimes F$ is semisimple for arbitrary finite extension field $F$ of $K$,
and so $Spec[A\otimes F]$ is a reduced $0$-dimensional scheme.
In addition, it is clear that the swich matrices of $A\otimes F$ are the same as those of $A$ and thus all of their determinants are invertible.
Therefore, $A\otimes F$ is semisimple and this completes the proof.
\end{proof}

Note that when $K$ is a perfect field, a finite dimensional affine $K$-algebra is
an \'{E}tale algebra if and only if it is reduced (see \cite[Remark 1.5.8]{Sz}).
This leads to a direct result as follows.

\begin{theorem}\label{4.16}
Let $K$ be a perfect field and $A$ an affine cellular algebra. Then $A$ is semisimple if and only if $A$ is separable.
\end{theorem}

\bigskip\bigskip

\noindent{\bf Acknowledgement}. 
%The authors would like to express their sincere thanks to the anonymous referee for
%her/his numerous helpful comments and corrections. 
The authors are grateful to Zeren Zheng for some helpful conversations.
Part of this work was done when Li visited Institute of Algebra and Number Theory 
at University of Stuttgart from August 2021 to September 2022. He takes this opportunity to express
his sincere thanks to the institute and Prof. S. Koenig for the hospitality during his visit.

\end{document}